\documentclass[11pt,a4paper,oneside]{amsart}
 \usepackage[a4paper]{geometry} 
\usepackage{amsmath} 
\usepackage{amssymb} 
\usepackage{amsthm} 
\usepackage{stmaryrd} 
\usepackage[english]{babel} 
\usepackage[font=small,justification=centering]{caption} 
\usepackage[nodayofweek]{datetime}
\usepackage[shortlabels]{enumitem} 
\usepackage[T1]{fontenc} 
\usepackage[utf8]{inputenc} 
\usepackage{ifthen} 
\usepackage{mathabx} 
\usepackage{mathtools} 
\usepackage{setspace} 
\usepackage{tikz-cd} 
\usepackage{xfrac} 
\usepackage[capitalize]{cleveref} 

\makeatletter
\@namedef{subjclassname@1991}{Mathematical subject classification 1991}
\@namedef{subjclassname@2000}{Mathematical subject classification 2000}
\@namedef{subjclassname@2010}{Mathematical subject classification 2010}
\@namedef{subjclassname@2020}{Mathematical subject classification 2020}
\makeatother

\newtheorem{thm}{Theorem}[section]
\newtheorem{lemma}[thm]{Lemma}
\newtheorem{corollary}[thm]{Corollary}
\newtheorem{prop}[thm]{Proposition}
\newtheorem{conjecture}[thm]{Conjecture}

\newtheorem{thmx}{Theorem}
\newtheorem{corx}[thmx]{Corollary}

\theoremstyle{definition}

\newtheorem{remark}[thm]{Remark}

\theoremstyle{plain}

    \newtheoremstyle{TheoremNum}
        {\topsep}{\topsep} 
        {\itshape} 
        {-0.25cm} 
        {\bfseries} 
        {.} 
        { }  
        {\thmname{#1}\thmnote{ \bfseries #3}}
    \theoremstyle{TheoremNum}

\newcommand*{\claimproofname}{My proof}

\DeclareMathOperator{\Aut}{\mathrm{Aut}}
\DeclareMathOperator{\Out}{\mathrm{Out}}

\newcommand{\cale}{{\mathcal{E}}}

\newcommand{\calg}{{\mathcal{G}}}

\newcommand{\calt}{{\mathcal{T}}}

\newcommand{\GL}{\mathrm{GL}}

\def\Z{\mathbb{Z}}

\newcommand{\NN}{\mathbb{N}}
\newcommand{\ZZ}{\mathbb{Z}}

\newcommand{\RR}{\mathbb{R}}

\usepackage{tikz}
\usetikzlibrary{arrows,quotes}
\tikzstyle{blackNode}=[fill=black, draw=black, shape=circle]

\title[Torsion homology growth]{Torsion homology growth of polynomially growing free-by-cyclic groups}

\usepackage[foot]{amsaddr}
\author{Naomi Andrew}
\author{Sam Hughes}
\author{Monika Kudlinska}
\address{Mathematical Institute, Andrew Wiles Building, Observatory Quarter, University of Oxford, Oxford OX2 6GG, UK}
\email{naomi.andrew@maths.ox.ac.uk}
\email{sam.hughes@maths.ox.ac.uk}
\email{kudlinska@maths.ox.ac.uk}
\date{\formatdate{10}{1}{2023}}


\begin{document}

\begin{abstract}
    We show that the homology torsion growth of a free-by-cyclic group with polynomially growing monodromy vanishes in every dimension independently of the choice of Farber chain.  It follows that the integral torsion $\rho^\mathbb{Z}$ equals the $\ell^2$-torsion $\rho^{(2)}$ verifying a conjecture of L\"uck for these groups.
\end{abstract}

\maketitle

\section{Introduction}
There are a number of notions of volume for a hyperbolic $3$-manifold $M$; namely, its hyperbolic volume $\mathrm{Vol}(M)$, the minimal volume entropy $\cale_{\rm min}(M)$, and the $\ell^2$-torsion $\rho^{(2)}(M)$. Combining results of L\"uck--Schick \cite{LueckSchick1999} for $\rho^{(2)}$, and Gromov \cite{Gromov1982}, Soma \cite{Soma1981}, and Thurston \cite{Thurston1978} for $\cale_{min}$, we see that for finite volume hyperbolic $3$-manifolds each of the invariants are proportional.  Conjecturally, (see \cite{BergeronVentktesh2013}) they are also proportional to a fourth invariant, the integral torsion
\[\rho^\Z (M)\coloneqq \sum_{j\geq0}t^{(2)}_j(M;M_n),\]
where
\[t^{(2)}_j(M;M_n)\coloneqq\limsup_{n\to\infty}\frac{\log |H_j(M_n;\Z )_{\mathrm{ tors}}|}{|M:M_n|}\]
is the torsion homology gradient with respect to a \emph{Farber sequence} $(M_n)_{n\in\NN}$ of finite covers.  See \Cref{sec rebuilding} for a definition of Farber sequence  --- the main example is a sequence of regular covers such that the intersection of the inclusions of the fundamental groups into $\pi_1 M$ is trivial.

Note that in particular, the conjecture states that the limit supremums are actually genuine limits and the limit is independent of the choice of Farber sequence.

\begin{conjecture}[L\"uck]\emph{\cite[Conjecture~1.11(3)]{Lueck2013}}
    Let $G$ be an infinite residually finite $\ell^2$-acyclic group of type $\mathsf{VF}$.  Then, $\rho^{(2)}(G)=\rho^\Z (G)$.
\end{conjecture}

L\"{u}ck proved the conjecture for fundamental groups of closed aspherical manifolds which admit non-trivial $S^1$-action, or where the group contains a non-trivial elementary amenable normal subgroup \cite[Corollary~1.13]{Lueck2013}. The conjecture also appears to be known for amenable groups (see \cite{KarKrophollerNikolov2017} for $\rho^{\Z}$ and \cite{LiThom2014} for $\rho^{(2)}$) and for groups acting on simplicial complexes with strict fundamental domain and with stabilisers containing normal abelian subgroups \cite{OkunSchreve2021}.  

Free-by-cyclic groups (understood to mean $\{$finitely generated free$\}$-by-$\{$infinite cyclic$\}$) are a natural algebraic generalisation of cusped $3$-manifolds which fibre over the circle.  However, due to their generality there is no notion of volume for them.  Due to the proportionality principle alluded to above it is tempting to take one of the other invariants $\rho^\Z$, $\rho^{(2)}$, or $\cale_{\mathrm min}$ as a definition of volume for free-by-cyclic groups.  Unfortunately computing these invariants is extremely difficult.  Indeed, there is a lower bound for $\cale_{\mathrm min}$ \cite{BregmanClay2021} (using \cite{BabenkoSabourau2020}) and for $\rho^{(2)}$ \cite{Clay2017} as well as some partial progress towards computing $\rho^{(2)}$ via a `chain flaring' condition \cite{Clay2021}.  There is an upper bound for $\rho^\Z$ for hyperbolic $3$-manifolds \cite{Le2018}. But for both hyperbolic $3$-manifolds and all free-by-cyclic groups with infinite order monodromy there are no explicit computations of $\rho^\Z$.

\begin{thmx}\label{thm main}
    Let $\Gamma=F_m\rtimes_\varphi\Z$.  If $\varphi$ is polynomially growing, then $t^{(2)}_j(\Gamma;\Gamma_n)=0$, independently of the choice of a Farber sequence $(\Gamma_n)_{n\in\NN}$, for all $j\geq 0$.
\end{thmx}

The theorem follows from establishing the \emph{cheap rebuilding property} and deducing the vanishing from the breakthrough work in \cite{AbertBergeronFraczykGaboriau2021}.  The cheap rebuilding property was introduced by Abert--Bergeron--Fraczyk--Gaboriau in loc. cit. and used to prove homology torsion growth vanishing for many families of groups.   We will give the background relevant to us below.
    
The next corollary follows from Clay's upper bound on the $\ell^2$-torsion of a free-by-cyclic group \cite[Theorem~5.1]{Clay2017}.
    
\begin{corx}
Let $\Gamma=F_m\rtimes_\varphi\Z$.  If $\varphi$ is polynomially growing, then $\rho^\Z (\Gamma)=\rho^{(2)}(\Gamma)=0$.
\end{corx}    

Combining the theorem with a result of Bregman--Clay \cite{BregmanClay2021} where they compute the minimal volume entropy of free-by-cyclic groups we obtain the following corollary.

\begin{corx}
Let $\Gamma=F_m\rtimes_\varphi\Z $.  If $\varphi$ is polynomially growing but $ \Gamma$ is not virtually tubular, then $\rho^\Z (\Gamma)\neq c\cdot\cale_{\rm min}(\Gamma)$ for any non-zero real-number $c$.
\end{corx}

Here $\Gamma$ being \emph{virtually tubular} implies there is a finite index subgroup $\Lambda\leq\Gamma$ such that $\Lambda$ splits as a graph of groups with vertex groups isomorphic to $\Z^2$ and edge groups is isomorphic to $\Z$.  Note that such a splitting implies that $\varphi$ is linearly growing.

It is tempting to conjecture a sharper relation between the degree of the monodromy and the homology torsion growth.

\begin{conjecture}
    Let $\Gamma=F_m\rtimes_\varphi\Z$.  If $\varphi$ is polynomially growing of degree $d$, then, 
    \[\lim_{n\to\infty} \frac{|H_1(\Gamma_n;\ZZ)_{\mathrm{tors}}|}{|\Gamma:\Gamma_n|^d}=c \]
    for some $c\in\RR_{\geq0}$ and any Farber sequence $(\Gamma_n)_{n\in\NN}$ of $\Gamma$.
\end{conjecture}

\subsection*{Structure of the paper}
In \Cref{sec FbyZ} we recall the necessary background on free-by-cyclic groups.  In \Cref{thm hierarchy} we give a self-contained proof of a splitting of a polynomially growing free-by-cyclic group with stabilisers of strictly lower polynomial growth (note that this result is well known to experts).  In \Cref{sec rebuilding} we give the relevant background from \cite{AbertBergeronFraczykGaboriau2021} which we will need to compute the homology torsion growth.  Finally, in \Cref{sec proofs} we prove the main theorem.

\subsection*{Acknowledgements}
This work has received funding from the European Research Council (ERC) under the European Union's Horizon 2020 research and innovation programme (Grant agreement No. 850930). The third author was supported by an Engineering and Physical Sciences Research Council studentship (Project Reference 2422910). 

The first and second authors would like to thank the organisers and funders of the Dame Kathleen Ollerenshaw Workshop ``Modern advances in geometric group theory'', held in Manchester during September 2022, where they began this work.  All three authors would like to thank Ric Wade for a number of helpful conversations and both Dawid Kielak and Armando Martino for carefully reading an earlier draft of this note.

The third author would like to thank her supervisors Martin Bridson and Dawid Kielak for continuous support during her DPhil.

\section{Background on free-by-cyclic groups}\label{sec FbyZ}

Let $F$ be a finite rank free group; we write $F_m$ if the rank $m$ is relevant. Let $\varphi$ be an element of $\Aut(F)$, representing an element $\Phi$ of $\Out(F)$. Throughout, we work with free-by-cyclic groups $F \rtimes_\varphi \ZZ$. Note that the free-by-cyclic groups defined by representatives of the same outer automorphism $\Phi$ are isomorphic, and so we are free to consider the group $F \rtimes_\Phi \ZZ$ without specifying a representative.

Fix a free set of generators of $F$. For any $g\in F$, we denote by $|g|$ the length of the reduced word representative of $g$. We write $|\bar{g}|$ to denote the minimal length of a cyclically reduced word representing a conjugate of $g$. 

An outer automorphism $\Phi \in \mathrm{Out}(F)$ acts on the conjugacy classes of elements in $F$. Given a conjugacy class $\bar{g}$ of an element $g\in F$, we say that $\bar{g}$ grows polynomially of degree $d$ under the iteration of $\Phi$, if there exist constants $C_1, C_2 > 0$ such that for all $n\geq 1$, \[C_1 n^d \leq |\Phi^n(\bar{g})| \leq C_2 n^d.\] Note that if $H \leq F$ is a subgroup whose conjugacy class in $F$ is preserved by $\Phi$, then for any $g \in H$ the order of growth of the conjugacy class $\bar{g}$ in $H$ is equal to the order of growth of the conjugacy class in $F$. We say $\Phi$ grows polynomially of degree $d$ if every conjugacy class of elements of $F$ grows polynomially of degree at most $d$ under the iteration of $\Phi$, and there exists a conjugacy class which grows polynomially of degree exactly $d$.

For a detailed discussion of the growth of elements or conjugacy classes under automorphisms or outer automorphisms see \cite{Levitt2009}. Note that by \cite[Lemma 2.3]{Levitt2009} and the subsequent example, there are linearly growing outer automorphisms for which there are elements growing quadratically (with respect to reduced word length) under certain representative automorphisms, so the choice to use conjugacy growth is necessary for our arguments.

An element $\Phi \in \mathrm{Out}(F_m)$ is said to be \emph{unipotent polynomially growing (UPG)}, if it is polynomially growing and it induces a unipotent element of $\mathrm{GL}(m, \mathbb{Z})$. It follows from \cite[Corollary 5.7.6]{BestvinaFeighnHandel00} that if $\Phi \in \Out(F)$ is polynomially growing, then it has a power $\Phi^k$ that is UPG; in fact $k$ can be taken to depend only on the rank of $F$. Note that the restriction of a UPG automorphism to an invariant free factor $H$ is again UPG.  The unipotence can be seen for instance by taking the restriction of the induced element of $\GL(m,\ZZ)$ to the image of $H$.

We will consider certain topological representatives $f$ of $\Phi$ on a graph $G$, and will need to understand growth of edges, paths and (cyclically reduced) loops $\gamma$ in $G$ under iteration of $f$: this is defined analogously to growth of conjugacy growth, but taking the lengths of reduced images $f^n(\gamma)$ in the path metric on $G$. (In fact, the definition of conjugacy growth can be recovered by representing $\Phi$ on a rose, and taking cyclically reduced loops representing each conjugacy class.)

By \cite[Theorem~5.1.8]{BestvinaFeighnHandel00}, any UPG outer automorphism $\Phi$ admits a relative train track representative $f \colon G \to G$ with the following additional properties. There exists an associated filtration $\emptyset = G_0 \subset G_1 \subset \ldots \subset G_n = G$, such that $G_{i+1}$ is constructed from $G_{i}$ by adding an oriented edge $E_{i+1}$. The map $f$ fixes every vertex of $G$, and for each $i \leq n$ there exists a closed path $\rho_i$ whose image lies in $G_{i-1}$, and such that $f(E_i) = E_i \cdot \rho_i$. In the language of \cite{BestvinaFeighnHandel00} this image is \emph{split}: $f^k(E_i \cdot \rho_i) = f^k(E_i) f^k(\rho_i)$, and no cancellation occurs between these images at any iterate. In particular, this means that no cancellation can occur between $f^k(\rho_i)$ and $f^{k+1}(\rho_i)$. If $\rho_i$ is non-constant, we may also assume that it is immersed, and we can assume that distinct $E_i$ have distinct suffixes $\rho_i$  (\cite[Remark 3.12]{BestvinaFeighnHandel05}). We call any such relative train track representative an \emph{improved relative train track}.

The following facts about improved relative train tracks representing UPG outer automorphisms are established by combining \cite[Lemmas 4.1.4 and 5.5.1]{BestvinaFeighnHandel00} (see also \cite[Lemma 6.5]{Levitt2009}). A \emph{Nielsen path} is a path that is fixed under iterating $f$ (after \emph{tightening} to remove backtracks as necessary). An \emph{exceptional path} is a path of the form $E_i \tau^k \overline{E}_j$, where $k$ is an integer, $\tau$ is a loop which is a Nielsen path (one can assume it is not a proper power), and $\rho_i, \rho_j$ are both positive powers of $\tau$. Exceptional paths do not split; under the assumption that distinct edges have distinct suffixes they are Nielsen paths if $i=j$ and have linear growth otherwise.

\begin{lemma}
    \label{BFH facts}
	Let $f$ be an improved relative train track representative for a UPG automorphism, and let $u$ be a path or loop in $G_k$. Then there is a power $M$ so that $f^M_\#(u)$ splits into $E_k$, $\overline{E}_k$, paths $\gamma$ contained in $G_{k-1}$, and exceptional paths (which may be Nielsen paths). If $E_k$ grows at least quadratically then the exceptional paths do not occur.
\end{lemma}

The following result is true in the more general setting of \emph{Kolchin maps}, however to avoid introducing extra terminology we will restrict to the case of improved relative train tracks, and indicate a proof using this technology.

\begin{lemma}\emph{\cite[Lemma~2.16]{Macura02}}\label{Macura}
Let $f \colon G \to G$ be an improved relative train track representative of an UPG outer automorphism. Let $E_i$ be an edge of $G$ which grows polynomially of degree $d \geq 2$, and suppose that $f(E_i) = E_i \cdot \rho_i$. Then every edge in $\rho_i$ grows polynomially of degree at most $d -1$, and there exists an edge which grows polynomially of degree $d - 1$.
\end{lemma}

Note that this lemma fails for linear growth: paths containing linearly growing edges can be fixed, and so $E_i$ may be linearly growing despite mapping over another such edge.
\begin{proof}

	That the bound is attained (including the ``moreover'' statement) follows since a finite difference argument gives that if $u_n$ grows polynomially of degree $d$, then $f^\ell_\#(E_n) = E_n \cdot u_n \cdot f_\#(u_n) \cdot f^2_\#(u_n) \cdot \ldots \cdot f^{\ell-1}_\#(u_n)$ grows polynomially with degree $d+1$, and the growth of $u_n$ is bounded above by the growth of its edges, and the paths in its splitting(s).
	
	The proof of the first claim is by induction on height in the filtration. The base case is the first $i$ in the filtration for which $E_i$ is at least quadratically growing; this cannot map over any other such edge by definition.
	
	We need a fairly strong inductive hypothesis. Suppose that the claim is true for every edge $E_i$ with height less than $n$, and note that this means that the filtration can be rearranged so that the filtration on $G_{n-1}$ respects the degrees of growth (in the sense that the top edge may be taken to have the fastest growth). We now show that if $E_n$ maps over an edge known to have polynomial growth of degree at least $2$, then $E_n$ has higher growth.
	
	Consider $E_n$ and suppose that $E_k$ grows polynomially with degree $d \geq 2$ and is the topmost edge in $u_n$. From the inductive hypothesis, we can safely assume that $E_k$ has the fastest growth of edges in $u_n$. \Cref{BFH facts} implies that for some (perhaps very large) $M$, $f^M_\#(u_n)$ splits as edges $E_k$ or $\overline{E}_k$, and paths strictly lower in the filtration. On further iteration of $f$, the edges $E_k$ and $\overline{E}_k$ grow polynomially with degree $d$; in particular so does any sufficiently high iterate of $u_n$. So $E_n$ grows polynomially with degree $d+1$.
\end{proof}

We now tie together the two notions of growth: that of edges under an improved relative train track representing $\Phi$, and that of conjugacy classes under $\Phi$.

\begin{lemma}
    Any non-trivial UPG outer automorphism $\Phi$ of $F_n$ enjoys the same degree of polynomial growth as the fastest growing edge in an improved relative train track representing $\Phi$.
\end{lemma}

This fact is certainly used implicitly in the literature, but we are unaware of an explicit proof and so we include one here for completeness.

\begin{proof}
    Note that since the universal covers of $G$ and the rose with $n$ petals are quasi-isometric, the growth of a conjugacy class under $\Phi$ is equal to the growth of any cyclically reduced loop in $G$ realising it.

    The growth of such a loop is bounded above by the growth of its edges. We will see that (provided $\Phi$ is non-trivial) it is always possible to construct a loop growing polynomially with the same degree. By \cref{Macura}, we can assume that the top edge in the filtration on $G$ has the highest growth (for the linear case, we might not have free choice as to which linearly growing edge is topmost, but one can certainly ensure that all linearly growing edges appear after all fixed edges).
    
    There are two cases to consider, with some additional subtleties in the case of linear growth. First suppose the top edge is non-separating, so there is a loop $u=E_n \gamma$, where $\gamma$ lives in $G_{n-1}$ (this loop is cyclically reduced, and so are its iterates). By \cref{BFH facts}, eventually some iterate splits either as $E_n \cdot \gamma'$, or with an exceptional path $E_n \tau^k E_i \cdot \gamma'$ (here $i < n$, and $\gamma'$ could be trivial). In the first case $u$ grows at least as $E_n$ does; in the second case both $E_n$ and $E_n \tau^k \overline{E}_i$ are linearly growing and again this bounds the growth of $u$ below.
    
    Now suppose the top edge is separating, so we must consider loops of the form $E_n \gamma_1 \overline{E}_n \cdot \gamma_2$, where $\gamma_1,\gamma_2$ are non-trivial loops. This loop is split as indicated (between $\overline{E}_n$ and $\gamma_2$), and as before all iterates remain cyclically reduced. We apply \cref{BFH facts}: if eventually some iterate of $u$ splits either as $E_n \cdot \gamma_1' \cdot \overline{E}_n \cdot \gamma_2'$, or with one exceptional path and one edge (either as $E_n \tau^k \overline{E}_i \cdot \gamma_1' \cdot \overline{E}_n \cdot \gamma_2'$ or as $E_n \cdot \gamma_1' \cdot E_j \tau^\ell \overline{E}_n \cdot \gamma_2'$), then $u$ grows as $E_n$ does and we are done.  Otherwise, $E_n$ is linearly growing and both $E_n$ and $\overline{E}_n$ form part of exceptional paths in the splitting eventually provided by \cref{BFH facts}. If this splitting is of the form $E_n \tau^k \overline{E}_i \cdot \gamma_1' \cdot E_j \tau^\ell \overline{E}_n \cdot \gamma_2'$, then as before $E_n \tau^k \overline{E}_i$ is linearly growing and this gives a lower bound for the growth of $u$.
    
    The remaining possibility is that $E_n \gamma_1 \overline{E_n}$ is already an exceptional path. Provided the connected component containing $\gamma_1$ is not a single cycle, one can choose some other $\gamma_1$ so that this does not occur. If this is not possible, $E_n$ does not map over another linearly growing edge, so the filtration can be rearranged so $E_n$ is below all other linearly growing edges. Provided there are linearly growing edges that do not separate a single cycle, such an edge can be assumed to be topmost, and one of the arguments above will go through. If all linearly growing edges are of this form, then the outer automorphism induced by this topological representative is trivial.
\end{proof}

\begin{remark}
    The kind of pathological representative of the trivial automorphism described at the end of this proof does seem to be a valid improved relative train track, though it would be outlawed by various improvements on this technology -- we chose to simplify our exposition by not taking these extra definitions and properties.
\end{remark}

The following proposition is well known to the experts (see for instance \cite{Macura02, Hagen2019}, and also \cite[Theorem 4.22]{BestvinaFeighnHandel05}). We include the proof here for completeness; 

\begin{prop}\label{thm hierarchy}
Let $\Phi \in \mathrm{Out}(F_m)$ be an UPG outer automorphism with polynomial growth of degree $d \geq 2$ and let $\Gamma = F_m \rtimes_{\Phi} \mathbb{Z}$. Then, $\Gamma$ splits as a finite graph of groups with edge groups isomorphic to $\mathbb{Z}$, and such that every vertex group is either conjugate to an incident edge group, or is isomorphic to $F_{m_v} \rtimes_{\Psi} \mathbb{Z}$ where $\Psi \in \mathrm{Out}(F_{m_v})$ is an UPG outer automorphism with growth of degree at most $d -1$, and $m_v < m$.
\end{prop}

\begin{proof}
Suppose $f \colon G \to G$ is an improved relative train track representative of the UPG outer automorphism $\Phi \in \mathrm{Out}(F_m)$ with growth of degree $d \geq 2$. We let $\mathcal{H}$ be the set of edges of $G$ with growth of degree $d$, and $G'$ be the subgraph of $G$ obtained by removing the interiors of the edges in $\mathcal{H}$. Note that $f(G') \subseteq G'$, since every edge in $G'$ grows polynomially with degree strictly less than $d$ and thus by Lemma~\ref{Macura}, the image of each edge in $G'$ does not contain an edge which grows with order $d$. Let $\{G_{i}'\}_{i \in I}$ be the set of connected components of $G'$. Since $f$ maps each edge $E$ to an edge-path which traverses $E$, it follows that $f$ also preserves each $G_{i}'$. Let $J \subseteq I$ be the indexing set of the non-simply-connected components of $G'$. Note that if $G_k'$ is simply-connected then each edge of $G_k'$ is fixed by $f$, since a non-trivial immersed loop does not map into $G_k'$.

Edges growing with degree $d$ cannot appear in the image of any other edge, so we can assume that the top edge $E_n$ in the filtration associated to $f$ grows with degree $d$. Since $f$ is an improved relative train track, $G$ cannot have vertices of valence $1$, and so if the edge $E_n$ is separating the connected components of $G_{n-1}$ cannot be simply connected, and must have non-trivial fundamental groups with ranks strictly less than $m$. (If it is non-separating then $G_{n-1}$ consists of a single connected component whose rank is $m-1$.) Therefore no $G_i'$ carries the whole fundamental group $F_m$.

Consider the graph of groups $\calg$ obtained from $G$ by collapsing each $G_i'$ to a single vertex. The Bass--Serre covering tree $\calt$ of $\calg$ is an $F_m$-tree and thus gives rise to a splitting of $F_m$, where the vertex stabilisers correspond to the fundamental groups of the components $G_i'$, and the edge stabilisers are trivial. Furthermore, since $f$ fixes every vertex and edge of $\calg$, the $F_m$-tree $\calt$ is preserved by $\Phi$. Hence there is an action of $\Gamma = F_m \rtimes_{\Phi} \mathbb{Z}$ on $\calt$. (This follows from {\cite[Theorem 3.7]{CullerMorgan1987}} or {\cite[Theorem 7.13(b)]{AlperinBass1987}}, since the pre-image of $\langle \Phi \rangle$ in $\Aut(F_m)$ preserves the translation length function of the $F_m$ action on $\calt$.) 

The vertex stabilisers are given by conjugates of \hbox{$\pi_1(G_i') \rtimes_{\varphi_i} \mathbb{Z}$}, where $\varphi_i$ is the restriction of an appropriate representative of $\Phi$ to the subgroup $\pi_1(G_i') < \pi_1(G) \simeq F_m$, for every $i \in I$. Note that if $G_i'$ is not simply connected, then $\pi_1(G_i')$ is a proper free factor of $\pi_1(G)$ and thus $\varphi_i$ represents a UPG outer automorphism  $\Phi_i$ of $\pi_1(G_i')$, and $\mathrm{rank}(\pi_1(G_i')) < m$. Furthermore, since the edges of the components $G_i'$ grow polynomially of degree at most $d- 1$ under the action of $f$, it follows that the $\Phi_i$ have polynomial growth of degree at most $d -1$. The edge stabilisers are infinite cyclic groups generated by (conjugates of) the stable letter of the original presentation as a semidirect product.
\end{proof}

We also record a similar decomposition theorem in the linearly growing case, obtained by the first author and Martino:

\begin{prop}\emph{\cite[Proposition~5.2.2]{AndrewMartino2022}}\label{linear splitting}
    Let $\Lambda=F_m\rtimes_\Phi\Z $.  Suppose that $\Phi$ is unipotent and linearly growing.  Then, $\Lambda$ splits as a finite graph of groups such that the vertex groups are isomorphic to $F_{m_v}\times\Z $ with $m_v\leq m$ and the edge groups are isomorphic to $\Z ^2$.
\end{prop}

\section{Affordable housing}\label{sec rebuilding}
Let following definition is due to Farber \cite{Farber1998}.  Let $\Gamma$ be a countable group and let $\mathrm{Sub}_\Gamma$ denote the space of subgroups of $\Gamma$ equipped with the induced topology from $\{0,1\}^\Gamma$ equipped with the topology of pointwise convergence.  The subspace $\mathrm{Sub}^{\mathrm{fi}}_\Gamma$ consisting of finite index subgroups of $\Gamma$ is countable if $\Gamma$ is finitely generated.  Note that $\Gamma$ acts continuously on both spaces.  For $\gamma\in\Gamma$ define the \emph{fixed point ratio function} as
\[\mathrm{fx}_{\Gamma,\gamma}\colon\mathrm{Sub}_{\Gamma}^{\mathrm{fi}}\to [0,1] \quad \text{by} \quad \Gamma'\mapsto \frac{|\{g\Gamma'|\gamma g\Gamma'=g\Gamma'\}|}{|\Gamma:\Gamma'|}. \]
A sequence $(\Gamma_n)_{n\in\NN}$ of subgroups of $\Gamma$ is a \emph{Farber sequence} if it consists of finite index subgroups and for every $\gamma\in\Gamma\backslash\{1\}$ we have \[\lim_{n\to\infty}\mathrm{fx}_{\Gamma,\gamma}(\Gamma_n) = 0.\]

To compute the homology gradients of Farber sequences we will use the \emph{cheap $\alpha$-rebuilding} property of Abert--Bergeon--Fraczyk--Gaboriau \cite{AbertBergeronFraczykGaboriau2021}.  The actual definition of this property is technical and need not concern us --- instead we will use a combination theorem from loc. cit..  The relevance of the $\alpha$-rebuilding property for us is the following deep theorem.

\begin{thm}[Abert--Bergeron--Fraczyk--Gaboriau] \emph{\cite[Theorem~H]{AbertBergeronFraczykGaboriau2021}}\label{rebuilding implies vanishing}
    Let $\Gamma$ be a countable group of type $\mathsf{F}_{\alpha+1}$ that has the cheap $\alpha$-rebuilding property for some $\alpha\in\NN$.  Then, for every Farber sequence $(\Gamma_n)$ of $\Gamma$ and $0\leq j\leq n$ we have $t^{(2)}_j(G;\Gamma_n)=0$.
\end{thm}

We will need the following lemmas.

\begin{lemma}[Abert--Bergeron--Fraczyk--Gaboriau] \emph{\cite[Corollary~10.13(1)]{AbertBergeronFraczykGaboriau2021}}\label{commensurate rebuilding}
Let $\alpha\in\NN$. Let $G$ be a residually finite group with finite index subgroup $H$.  Then $G$ has the cheap $\alpha$-rebuilding property if only if $H$ does.
\end{lemma}

\begin{lemma}[Abert--Bergeron--Fraczyk--Gaboriau] \emph{\cite[Corollary~10.13(2)]{AbertBergeronFraczykGaboriau2021}}\label{normal Z rebuilding}
    If a residually finite group $G$ of type $\mathsf{VF}$ has a normal subgroup isomorphic to $\Z$, then $G$ has the cheap $\alpha$-rebuilding property for all $\alpha\in\NN$.
\end{lemma}

At last, the promised combination theorem:

\begin{thm}[Abert--Bergeron--Fraczyk--Gaboriau] \emph{\cite[Theorem~F]{AbertBergeronFraczykGaboriau2021}}\label{rebuilding}
    Let $\Gamma$ be a residually finite group acting on a CW complex $X$ such that any element stabilising a cell fixes it pointwise.  Let $\alpha\in\NN$ and assume that the following conditions hold:
    \begin{enumerate}
    \item $X/\Gamma$ has finite $\alpha$-skeleton;
    \item $X$ is $(\alpha-1)$-connected;
    \item the stabiliser of every cell of dimension $j\leq \alpha$ has the cheap $(\alpha-j)$-rebuilding property.
    \end{enumerate}
    Then, $\Gamma$ itself has the cheap $\alpha$-rebuilding property.
\end{thm}

\section{Proof of the main theorem} \label{sec proofs}

\begin{prop}\label{prop main}
    Let $\Gamma=F_m\rtimes_\varphi\Z $ be a free-by-cyclic group with polynomially growing monodromy $\varphi$ of degree $d$.  Then, $\Gamma$ has the cheap $\alpha$-rebuilding property for all $\alpha\in\NN$.
\end{prop}

\begin{proof}
We proceed by induction on $d$. Note that for any $k \neq 0$, the proposition holds for $\varphi$ if and only if it holds for $\varphi^k$. Indeed, the group $F_m\rtimes_{\varphi^k}\Z $ is a finite index subgroup of $\Gamma$, and by \Cref{commensurate rebuilding}, the cheap $\alpha$-rebuilding property is a commensurability invariant. Therefore for all $d$ we begin by passing to a unipotent power (when $d=0$, a unipotent power is inner), and carry out the induction assuming $\Phi$ is UPG.

\paragraph{\textbf{Base cases:}} \emph{periodic and linearly growing.}

If $d=0$, then $F_m \rtimes_{\varphi} \mathbb{Z} \cong F_m \times \mathbb{Z}$. By \Cref{normal Z rebuilding} such a group has the cheap $\alpha$-rebuilding property for all $\alpha \in \NN$.

Suppose that $d = 1$.
By \Cref{linear splitting}, $\Gamma$ splits as a finite graph of groups such that the vertex groups are isomorphic to $F_{k_v}\times\Z$ with $k_v\leq n$ and the edge groups are isomorphic to $\Z^2$.  Note that this implies the stabilisers of $\Gamma$ on the Bass-Serre tree $\calt$ of this splitting are all isomorphic to $F_{k_v}\times\Z $ or $\Z ^2$ with $k_v\leq m$.  Observe that by \Cref{normal Z rebuilding} the groups $F_{k_v}\times \Z $ and $\Z ^2$ have the cheap $\alpha$-rebuilding property for all $\alpha\in\NN$.  The result follows from applying \Cref{rebuilding} to $\calt$.\hfill$\blackdiamond$

\paragraph{\textbf{Induction step:}} \hfill

We now assume that the growth $d$ of the monodromy $\varphi$ is at least quadratic and that any free-by-cyclic group with monodromy polynomially growing of degree at most $d-1$ has the cheap $\alpha$-rebuilding property for all $\alpha\geq0$.
We continue to assume $\varphi$ is unipotent.

It follows from \Cref{thm hierarchy} that $\Gamma$ splits as a graph of groups such that the vertex groups are free-by-cyclic groups with polynomially growing monodromy of degree at most $d-1$ and the edge groups are isomorphic to $\Z$.   
Now, $\Gamma$ acts cocompactly on the Bass-Serre tree $\calt$ of this graph of groups.  
By the induction hypothesis the stabilisers of this action admit the cheap $\alpha$-rebuilding property for all $\alpha\in\NN$. \hfill$\blackdiamond$

This completes the proof of the proposition.
\end{proof}

\begin{proof}[Proof of \Cref{thm main}]
    This follows immediately from \Cref{prop main} and \Cref{rebuilding implies vanishing}.
\end{proof}

\bibliographystyle{alpha}
\bibliography{refs.bib}

\end{document}